\documentclass[a4paper,10pt]{article}

\usepackage{amssymb,amsthm} 
\usepackage{xypic}  
\usepackage{amsmath} 
\usepackage{mathrsfs}

\def\R{{\mathbb R}}
\def\C{{\mathbb C}}

\def\A{{\mathcal A}}

\def\epsi{\varepsilon}

\def\G{{\mathscr G}}
\def\BG{{\mathscr B}}
\def\H{{\mathscr H}}
\def\K{{\mathscr K}}

\def\d{{\mathrm d}}

\newtheorem{theorem}{Theorem}[section]

\newtheorem{lemma}[theorem]{Lemma}
\newtheorem{proposition}[theorem]{Proposition}
\newtheorem{corollary}[theorem]{Corollary}

\theoremstyle{definition}
\newtheorem{definition}[theorem]{Definition}

\newtheorem{example}[theorem]{Example}

\theoremstyle{remark}

\numberwithin{equation}{section}

\title{Secondary Theories for \'etale groupoids}

\author{Marcello Felisatti \thanks{\noindent Escuela de Matematica, Universidad
de Costa Rica \protect \\ \protect
Recinto Rodrigo Facio, San Jos\'e, apartado 2060  Costa Rica \protect \\
\protect mfelisatti@gmail.com}
\and Frank Neumann \thanks{Department of Mathematics, University of Leicester \protect \\
University Road, Leicester LE1 7RH, England, UK \protect \\ fn8@mcs.le.ac.uk}}

\begin{document}

\maketitle

\begin{abstract}
{Generalizing Karoubi's multiplicative K-theory and multiplicative cohomology
groups for smooth manifolds we define secondary theories and characteristic
classes for smooth \'etale groupoids. As special cases we obtain versions of the
groups of differential characters 
for smooth \'etale groupoids and for orbifolds.}
\end{abstract}

\section*{Introduction}

In this article we introduce secondary theories and characteristic classes for
principal bundles with connections over \'etale groupoids. 
In particular, we generalize the multiplicative cohomology groups of Karoubi
\cite{K1}, \cite{K2} and the groups of Cheeger-Simons differential characters
\cite{CS} for smooth manifolds to smooth \'etale groupoids and arbitrary
filtrations of the associated simplicial de Rham complex and analyse their
interplay. 
By specializing to Deligne's `filtration b\^{e}te' we obtain generalized
versions of Cheeger-Simons differential characters for \'etale groupoids. In the
particular case of a proper \'etale groupoid representing an orbifold we obtain
orbifold versions of these secondary theories, which in the case of differential
characters were studied before by Lupercio and Uribe \cite{LU} using the
geometric approach of Hopkins and Singer \cite{HS} for the construction of
generalized differential cohomology theories. 

More generally we are giving a definition of generalized differential characters
for smooth \'etale groupoids associated to an arbitrary filtration of the de
Rham complex, which in the case of the `filtration b\^{e}te' gives the classical
versions of Cheeger and Simons. This more general definition allows to construct
an explicit map on the levels of cocycles between multiplicative cohomology and
these generalized differential characters. It follows that these general groups
of multiplicative cohomology are the natural places for constructing and
analysing secondary characteristic classes for principal bundles with
connections on \'etale groupoids. In general such principal bundles, in contrast
with the case of smooth manifolds, might not always admit connections, but in the
important case for example of  principal bundles over proper \'etale groupoids,
which are groupoids representing orbifolds, such connections always exist. 

In the first section we will give an overview of the basic elements of de Rham
and Chern-Weil theory for smooth \'etale groupoids following \cite{LTX}.
In the following section we define multiplicative cohomology and generalized
differential characters for smooth \'etale groupoids and study their main
properties. We will make explicit use here of the simplicial machinery as
developed in \cite{FN}, but for the convenience of the reader we will recall
many of the details. This was in part influenced by the work of Dupont
\cite{D1}, \cite{D2} and Dupont-Hain-Zucker \cite{DHZ}. 
In the third section we introduce multiplicative bundles and versions of
multiplicative $K$-theory for smooth \'etale groupoids. In the final section we
then construct characteristic classes of elements in multiplicative $K$-theory
for smooth \'etale groupoids with values in the groups of multiplicative
cohomology and generalized differential characters. 

In a sequel to this article we aim to study the relationship between our
generalized groups of multiplicative cohomology and smooth Deligne cohomology
for \'etale groupoids and to construct in a unifying way secondary theories and
characteristic classes for principal bundles over orbifolds, foliations and
differentiable stacks. Versions of smooth Deligne cohomology for particular
\'etale groupoids were also studied before, for example in the case of
particular orbifold groupoids by Lupercio-Uribe \cite{LU} and for general
transformation groupoids by Gomi \cite{Go}. It will be interesting to analyse if
these secondary theories and characteristic classes could be also defined in the
context of algebraic geometry, for example in order to extend algebraic
differential characters as introduced and studied by Esnault \cite{E1},
\cite{E2} to smooth Deligne-Mumford stacks. Real Deligne cohomology 
groups of proper smooth Deligne-Mumford stacks over the complex numbers $\C$
appear also in the context of arithmetic intersection theory on Deligne-Mumford
stacks as introduced recently by Gillet \cite{Gi}.

\medskip \noindent {\it Acknowledgements.} This work was initiated while the
second author was visiting the Tata Institute of Fundamental Research in Mumbai.
He likes to thank Prof. N. Nitsure for the kind invitation. He also acknowledges
financial support by the London Mathematical Society and the University of
Leicester.

\section{Elements of  Chern-Weil theory  for \'etale \\ groupoids}\label{sim}

Let us recall the main ingredients of de Rham and Chern-Weil theory for \'etale
groupoids. We will be following here the simplicial approach of  \cite{DHZ} and
\cite{FN}. For the general theory of \'etale groupoids and their Chern-Weil
theory we refer the reader to \cite{LTX} and \cite{CM1}. We will restrict
ourselves in this note to the case of \'etale groupoids, but many of the
concepts and constructions can be derived also in more general contexts. 

Let $\G=[\G_1\rightrightarrows^{\hspace*{-0.25cm}^{s}}_{\hspace*{-0.25cm}{_t}}
\G_0]$ be a {\it smooth \'etale groupoid}, this means that all structure maps of
$\G$
$$\G_1 \times_{t, \G_0, s} \G_1 \stackrel{m}\rightarrow
\G_1\stackrel{i}\rightarrow
\G_1\rightrightarrows^{\hspace*{-0.25cm}^{s}}_{\hspace*{-0.25cm}{_t}}
\G_0\stackrel{e}\rightarrow \G_1 $$
are local diffeomorphisms with $\G_0, \G_1$ smooth manifolds, $m$ the
composition map, $i$ the inverse, $e$ the identity and $s, t$ the source and
target maps.

If in addition, the anchor map
$$(s, t): \G_1\rightarrow \G_0\times \G_0$$
is {\it proper}, the groupoid $\G$ is called a {\it proper} smooth \'etale
groupoid. Each proper smooth \'etale groupoid represents an orbifold \cite{M}
and therefore the concepts and constructions presented in this note have direct
applications to orbifolds. 

Associated to any smooth \'etale groupoid
$\G=[\G_1\rightrightarrows^{\hspace*{-0.25cm}^{s}}_{\hspace*{-0.25cm}{_t}}
\G_0]$ is a canonical smooth simplicial manifold $\G_{\bullet}=\{\G_n\}$, where
for each $n\geq 0$
$$\G_n=\G_1 \times_{t, \G_0 ,s} \G_1 \cdots \times_{t, \G_0, s} \G_1 \,\,\,
\mathrm{(n\,\, factors)}$$
is the smooth manifold consisting of all composable $n$-tuples together with the
canonical smooth face and degeneracy maps
$$\epsi_i: \G_n\rightarrow \G_{n-1}, \,\,\, \eta_i: \G_n\rightarrow \G_{n+1}$$
for $0\leq i\leq n$ such that the usual simplicial identities hold (see
\cite{S}, \cite{D1}). In other words, $\G_{\bullet}$ is a simplicial object in
the category of smooth manifolds.

To the simplicial smooth manifold $\G_{\bullet}$ we can furthermore associate a
simplicial space $||\G_{\bullet}||$, its {\it fat realization}, defined as the 
quotient space
$$||\G_{\bullet}||= \coprod_{n\geq 0} (\Delta^n \times \G_n)/\sim$$
where the equivalence relation is generated by
$$(\epsi^i \times id)(t,x)\sim (id \times \epsi_i)(t,x)$$
for any $(t, x) \in \Delta^{n-1} \times \G_n$.

A fundamental concept for \'etale groupoids is that of Morita equivalence.
Basically a Morita equivalence class of \'etale groupoids defines a
differentiable stack. Let us very briefly recall the main constructions here. We
start with the definition of generalized homomorphisms between \'etale
groupoids.

\begin{definition}
A {\it generalized homomorphism} $\phi:=(Z, \sigma, \tau)$ between \'etale
groupoids  $\G=[\G_1\rightrightarrows \G_0]$ and $\H=[\H_1\rightrightarrows
\H_0]$ is given by a smooth manifold $Z$, two smooth maps
$\G_0\stackrel{\tau}\leftarrow Z \stackrel{\sigma}\rightarrow \H_0$, a left
action of $\G_1$ with respect to $\tau$, a right action of $\H_1$ with respect
to $\sigma$, such that the two actions commute, and $Z$ is an $\H_1$-principal
bundle over $\G_0$.
\end{definition}

Two generalized homomorphisms $\phi_1:=(Z_1, \sigma_1, \tau_1)$ and
$\phi_2:=(Z_2, \sigma_2, \tau_2)$ from $\G$ to $\H$ are called {\it equivalent}
if there is an $\G_1$-$\H_1$-equivariant diffeomorphism $\psi: Z_1\rightarrow
Z_2$. Composition of generalized homomorphisms is defined as follows: if $\phi:
\G_0\stackrel{\tau}\leftarrow Z \stackrel{\sigma}\rightarrow \H_0$ is a
generalized homomorphism from 
$\G=[\G_1\rightrightarrows \G_0]$ to $\H=[\H_1\rightrightarrows \H_0]$ and
$\phi': \H_0\stackrel{\tau'}\leftarrow Z' \stackrel{\sigma'}\rightarrow \K_0$
a generalized homomorphism from $\H=[\H_1\rightrightarrows \H_0]$ to
$\K=[\K_1\rightrightarrows \K_0]$, then the composition
$\phi\circ \phi': \G_0\leftarrow Z'' \rightarrow \K_0$ defined by
$$Z''=Z\times_{\H_1} Z':=(Z\times_{\sigma, \H_0, \tau'} Z')/{(z, z') \sim (zh,
h^{-1}z')}$$
is a generalized homomorphism from $\G=[\G_1\rightrightarrows \G_0]$ to
$\K=[\K_1\rightrightarrows \K_0]$. The composition of equivalence classes of
generalized homomorphisms is also associative and we get the {\it category
${\mathcal Grp}$ of \it \'etale groupoids}, whose objects are \'etale groupoids
and whose morphisms are generalized homomorphisms. The isomorphisms in this
category $\mathcal Grp$ are called {\it Morita equivalences} and it follows that
any generalized homomorphism of \'etale groupoids can be decomposed as the
composition of a Morita equivalence and a strict homomorphism of groupoids.
The category ${\mathcal Grp}[{\mathcal M}]^{-1}$ obtained from $\mathcal Grp$
via localizing with respect to Morita equivalences gives the {\it category of
differentiable stacks}, i.e. a differentiable stack can be thought of as a
Morita equivalence class of an \'etale groupoid \cite{Pr}.

Let us now recall the basic ingredients of de Rham and Chern-Weil theory for
\'etale groupoids using the approach in  \cite{DHZ} and \cite{FN} for simplicial
smooth manifolds. Working simplicially, there are in fact two versions of  the
de Rham and singular complexes we can associate to an \'etale groupoid
$\G=[\G_1\rightrightarrows^{\hspace*{-0.25cm}^{s}}_{\hspace*{-0.25cm}{_t}}
\G_0]$ using the associated simplicial smooth manifold $\G_{\bullet}$.\\

{\bf The de Rham complex of compatible forms.} A simplicial smooth
complex $k$-form $\omega$ on $\G_{\bullet}$ is a sequence
$\{\omega^{(n)}\}$ of smooth complex $k$-forms $\omega^{(n)}\in
\Omega^k_{dR}(\Delta^n\times \G_n)$ satisfying the compatibility
condition
$$(\epsi^i\times id)^* \omega^{(n)}=(id \times \epsi_i)^*\omega^{(n-1)}$$
in $\Omega^k_{dR}(\Delta^{n-1}\times \G_n)$ for all $0\leq i\leq n$
and all $n\geq 1$. Let $\Omega^k_{dR}(\G_{\bullet})$ be the set of
all simplicial smooth complex $k$-forms on $\G_{\bullet}$. The
exterior differential on $\Omega^k_{dR}(\Delta^n\times \G_n)$ induces
an exterior differential $\d$ on $\Omega^k_{dR}(\G_{\bullet})$. We
denote by $(\Omega^*_{dR}(\G_{\bullet}), \d)$ the de Rham complex of
compatible forms.

We note that $(\Omega^*_{dR}(\G_{\bullet}), \d)$ is given as the total complex
of a double complex $(\Omega^{*,*}_{dR}(\G_{\bullet}),\d', \d'')$
with
$$\Omega^k_{dR}(\G_{\bullet})=\bigoplus_{r+s=k}
\Omega^{r,s}_{dR}(\G_{\bullet})$$
and $\d=\d'+ \d''$, where $\Omega^{r,s}_{dR}(\G_{\bullet})$ is the
vector space of $(r+s)$-forms, which when restricted to
$\Delta^n\times \G_n$ are locally of the form
$$\omega|_{\Delta^n\times \G_n}= \sum a_{i_1\ldots i_r j_1\ldots j_s}\d
t_{i_1}\wedge
\ldots \d t_{i_r}\wedge \d x_{j_1}\wedge \d x_{j_s},$$
where $(t_0, \ldots, t_n)$ are barycentric coordinates of $\Delta^n$ and
the $\{x_j\}$ are local coordinates of $\G_n$. Furthermore the differentials
$\d'$ and $\d''$ are
the exterior differentials on $\Delta^n$ and $\G_n$ respectively.

Note that $\omega=\{\omega^{(n)}\}$ gives a smooth $k$-form on
$$\coprod_{n\geq 0} (\Delta^n \times \G_n)$$ and the compatible condition is
precisely what is needed to define a form on the fat realization
$||\G_{\bullet}||$ of $\G_{\bullet}$.\\

{\bf The simplicial de Rham complex.} The de Rham complex
$(\A^*_{dR}(\G_{\bullet}), \delta)$ of $\G_{\bullet}$ is given as the
total complex of a double complex $(\A^{*,*}_{dR}(\G_{\bullet}),
\delta', \delta'')$ with
$$\A^k_{dR}(\G_{\bullet})=\bigoplus_{r+s=k} \A^{r,s}_{dR}(\G_{\bullet})$$
and $\delta =\delta' + \delta''$, where
$\A^{r,s}_{dR}(\G_{\bullet})=\Omega^s_{dR}(\G_r)$ is the set of smooth
complex $s$-forms on the smooth manifold $\G_r$. Furthermore the
differential
$$\delta'': \A^{r,s}_{dR}(\G_{\bullet})\rightarrow \A^{r,
s+1}_{dR}(\G_{\bullet})$$
is the exterior differential on $\Omega^*_{dR}(\G_r)$ and the
differential
$$\delta':\A^{r,s}_{dR}(\G_{\bullet})\rightarrow \A^{r+1,
s}_{dR}(\G_{\bullet})$$
is defined as the alternating sum
$$\delta'=\sum_{i=0}^{r+1}(-1)^i \epsi_i^*.$$

{\bf The simplicial singular cochain complex.} Given a commutative ring $R$ we
can also associate a
singular cochain complex $(S^*(\G_{\bullet}; R), \partial)$ to $\G_{\bullet}$.
It is
defined as a double complex $(S^{*,*}(\G_{\bullet}; R), \partial',
\partial'')$ with
$$S^k(\G_{\bullet}; R)=\bigoplus_{r+s=k} S^{r,s}(\G_{\bullet}; R)$$
and $\partial=\partial'+\partial''$, where
$$S^{r,s}(\G_{\bullet};R)=S^s(\G_r; R)$$
is the set of singular cochains of degree $s$ on the smooth manifold $\G_r$.

There is an integration map
$${\mathcal I}: \A^{r,s}_{dR}(\G_{\bullet})\rightarrow S^{r,s}(\G_{\bullet};
\C)$$
which gives a morphism of double complexes and Dupont's general version of the
de Rham theorem
(see \cite{D2}, Proposition 6.1) shows that this integration map induces natural
isomorphisms
$$H^*(\A^{*}_{dR}(\G_{\bullet}, \delta))\cong H^*(S^*(\G_{\bullet},
\C),\partial)\cong H^*(||\G_{\bullet}||; \C).$$

From Stokes' theorem it follows that there is also morphism of complexes
$${\mathcal J}: \Omega^*_{dR}(\G_{\bullet}), \d) \rightarrow
(\A^*_{dR}(\G_{\bullet}), \delta)$$
defined on $\Omega_{dR}^*(\Delta^n\times \G_n)$ by integration over
the simplex $\Delta^n$
$$\omega^{(n)}\in \Omega_{dR}^*(\Delta^n\times \G_n) \mapsto \int_{\Delta^n}
\omega^{(n)}.$$
This morphism is in fact a quasi-isomorphism (see \cite{D1}, Theorem 2.3 and
Corollary 2.8), i.e. we have
$$H^*(\Omega^*_{dR}(\G_{\bullet}), \d)\cong H^*(\A^*_{dR}(\G_{\bullet}),
\delta)\cong H^*(||\G_{\bullet}||, \C).$$

{\bf The singular cochain complex of compatible cochains.} Let $R$
be a commutative ring. A compatible singular cochain $c$ on
$\G_{\bullet}$ is a sequence $\{c^{(n)}\}$ of cochains $c^{(n)}\in
S^k(\Delta^n\times \G_n; R)$ satisfying the compatibility condition
$$(\epsi^i\times id)^* c^{(n)}=(id \times \epsi_i)^* c^{(n-1)}$$
in $S^k(\Delta^{n-1}\times \G_n)$ for all $0\leq i\leq n$ and all
$n\geq 1$. Let $C^k(\G_{\bullet}; R)$ be the set of all compatible singular
cochains on $\G_{\bullet}$ and $(C^*(\G_{\bullet}; R), \d)$ be the singular
cochain complex of compatible cochains.

The natural inclusion of cochain complexes
$$(C^*(\G_{\bullet}; R), \d)\rightarrow (S^*(\G_{\bullet}; R), \partial)$$
is a quasi-isomorphism \cite{DHZ}.

Integrating forms preserves the compatibility conditions and therefore we get an
induced map of complexes \cite{DHZ}
$${\mathcal I'}: \Omega^*_{dR}(\G_{\bullet})\rightarrow C^*(\G_{\bullet}; \C)$$
fitting into a commutative diagram
\[
\xymatrix{ \Omega^*_{dR}(\G_{\bullet})\ar[r]^{\mathcal
I'}\ar[d]^{\mathcal J}&
C^*(\G_{\bullet}; \C)\ar[d]\\
A^*_{dR}(\G_{\bullet})\ar[r]^{\mathcal I}& S^*(\G_{\bullet}; \C) }
\]
and which is again a quasi-isomorphism, i.e. we have
$$H^*(\Omega^*_{dR}(\G_{\bullet}), \d)\cong H^*(C^*(\G_{\bullet}; \C),
\partial).$$

We call $H^*_{dR}(\G_{\bullet}):=H^*(\Omega^*_{dR}(\G_{\bullet}, \d))$ the {\it
de Rham cohomology} of the groupoid $\G$ and $H^*(\G_{\bullet},
R):=H^*(C^*(\G_{\bullet}; R), \partial)$ the {\it singular cohomology} of  $\G$.
 De Rham and singular cohomology behave well with respect to Morita equivalence
of  \'etale groupoids and are therefore well-defined invariants for
differentiable stacks (see \cite{B1}, Def. 9. or \cite{LTX}, 3.1).

\begin{proposition} Let $R$ be a commutative ring. If $\G=[\G_1\rightrightarrows
\G_0]$ and $\H=[\H_1\rightrightarrows \H_0]$ are \'etale groupoids which are
Morita equivalent, then we have isomorphisms
$$H^k_{dR}(\G_{\bullet})\stackrel{\simeq}\rightarrow H^k_{dR}(\H_{\bullet}),
\,\,\mathrm{and}\,\, H^k(\G_{\bullet}, R)\stackrel{\simeq}\rightarrow
H^k(\H_{\bullet}, R).$$
\end{proposition}
\begin{proof} That the simplicial de Rham and singular complexes of Morita
equivalent \'etale groupoids give isomorphic cohomology groups is well-known
(see for example \cite{TXL}, Prop. 2.15, \cite{B1} or \cite{CM1}). The above
quasi-isomorphisms between the respective complexes of simplicial and compatible
forms and cochains then give the desired result. \end{proof}
\begin{example}
If $\G=[M\rightrightarrows M]$ is just a smooth manifold $M$, then the de Rham
cohomology and singular cohomology groups of $\G$ are just the ones for smooth
manifolds. Indeed, if $\{U_i\}$ is an open covering of $M$ and $\coprod_i
U_i\rightarrow M$ the \'etale map, then the \'etale groupoid
$\Gamma=[\coprod_{i,j} U_i\cap U_j\rightrightarrows \coprod_i U_i]$ is Morita
equivalent to $\G$ and $H^*_{dR}(\Gamma_{\bullet})\cong H^*_{dR}(M)$. 
\end{example}
\begin{example}Let $G$ be a compact Lie group. If $\G=[G\times
M\rightrightarrows M]$ is the transformation groupoid, then
$H^*_{dR}(\G_{\bullet})$ and $H^*(\G_{\bullet}, R)$ are the $G$-equivariant de
Rham and Borel cohomology groups of $M$.
\end{example}
\begin{example} Let $\G=[\G_1\rightrightarrows \G_0]$ be a proper \'etale
groupoid representing an orbifold and let $\Lambda(\G)$ be the associated
inertia groupoid, i. e. the transformation groupoid
$\Lambda(\G)=[\Omega(\G)\rtimes \G \rightrightarrows \Omega(\G)]$ in which we
denote by $\Omega(\G)=\{g\in \G_1|s(g)=t(g)\}$ the manifold of closed loops.
Then $H^*_{dR}(\Lambda (\G)_{\bullet})$ and $H^*(\Lambda (\G)_{\bullet}, R)$ are
the associated orbifold cohomology groups.
\end{example}
In these examples we could have either used the simplicial or compatible
complexes of $\G_{\bullet}$. We will use the compatible de Rham and cochain
complexes to define generalized versions of  Karoubi's multiplicative cohomology
and Cheeger-Simons differential characters for a given \'etale groupoid $\G$.

Finally we recall the basic aspects of Chern-Weil theory for \'etale groupoids
as developed in \cite{LTX}.\\

{\bf Principal $G$-bundles over \'etale groupoids.} 
Let $G$ be a Lie group with Lie algebra $\frak{g}$ and
$\G=[\G_1\rightrightarrows^{\hspace*{-0.25cm}^{s}}_{\hspace*{-0.25cm}{_t}}
\G_0]$ be an \'etale groupoid. A {\it principal $G$-bundle} over $\G$ consists
of a (right) $G$-bundle $P\stackrel{\pi}\rightarrow \G_0$ over the smooth
manifold $\G_0$ such that the groupoid $\G$ acts on $P\stackrel{\pi}\rightarrow
\G_0$ and this action commutes with the $G$-action (see \cite{LTX}, Def. 2.2).

It turns out, that the category of principal $G$-bundles over an \'etale
groupoid 
$\G=[\G_1\rightrightarrows^{\hspace*{-0.25cm}^{s}}_{\hspace*{-0.25cm}{_t}}
\G_0]$ is equivalent to the category of principal $G$-bundles over the 
simplicial smooth manifold $\G_{\bullet}$ (see \cite{LTX}, Prop. 2.4). Here a
principal $G$-bundle over $\G_{\bullet}$ is given by a simplicial smooth
manifold $P_{\bullet}$ and a morphism $\pi_{\bullet}: P_{\bullet}\rightarrow
\G_{\bullet}$ of simplicial smooth manifolds, such that
\begin{itemize}
\item[(i)] for each $n$ the map $\pi_{p}: P_{n}\rightarrow \G_{n}$ is a
principal $G$-bundle
over $\G_n$
\item[(ii)] for each morphism $f: [m]\rightarrow [n]$ of the simplex category
$\Delta$ the induced map
$f^*: P_{n}\rightarrow P_{m}$ is a morphism of $G$-bundles, i.e. we have a
commutative
diagram
\[
\xy
\xymatrix{
P_n\ar[r]^{f^*}
\ar[d]&
P_m\ar[d]\\
\G_n\ar[r]^{f^*}&\G_m}
\endxy
\]
\end{itemize}
The last condition is just saying that the degeneracy and face maps are
morphisms of principal $G$-bundles.

It follows also, that if $\pi_{\bullet}: P_{\bullet}\rightarrow \G_{\bullet}$ is
a principal
$G$-bundle over $\G_{\bullet}$, then $|\pi_{\bullet}|: |P_{\bullet}|\rightarrow
|\G_{\bullet}|$ is a principal $G$-bundle with $G$-action induced by
$$\Delta^n\times P_{n}\times G\rightarrow \Delta^n\times P_{n}, \, \,
(t,x,g)\mapsto (t, xg).$$

It can be shown that the category of principal $G$-bundles over an \'etale
groupoid $\G=[\G_1\rightrightarrows \G_0]$ is equivalent to the category of
generalized homomorphisms from $[\G_1\rightrightarrows \G_0]$ to the groupoid
$[G\rightrightarrows \cdot]$ and therefore if $\G$ and $\G'$ are Morita
equivalent \'etale groupoids, there is an equivalence of categories of principal
$G$-bundles over $\G$ and $\G'$ (see \cite{LTX}, Prop. 2.11 and Cor 2.12). There
is also a well-defined notion of pull-back of principal $G$-bundles over \'etale
groupoids along generalized homomorphisms.\\

{\bf Connections and curvature for principal $G$-bundles.} Let $G$ be a Lie
group with Lie algebra $\frak{g}$. Furthermore let $\G=[\G_1\rightrightarrows
\G_0]$ be an \'etale groupoid and $P\stackrel{\pi}\rightarrow \G_0$ be a
principal $G$-bundle over $\G$. Following \cite{LTX}, 3.2 we consider
pseudo-connection and pseudo-curvature forms for principal $G$-bundles over
\'etale groupoids.
A {\it pseudo-connection} is a connection $1$-form $\theta\in
\Omega^1_{dR}(P)\otimes \frak{g}$ of the $G$-bundle $P\stackrel{\pi}\rightarrow
\G_0$
(ignoring the groupoid action). The {\it total pseudo-curvature} is a $2$-form 
$\Omega_{\mathrm{total}}\in \Omega^2_{dR}(Q_{\bullet})\otimes \frak{g}$ defined
as
$$\Omega_{\mathrm{total}}=\delta\theta + \frac{1}{2}[\theta, \theta]=
\partial\theta + \Omega$$
where $\Omega=d\theta +\frac{1}{2}[\theta, \theta]\in
\Omega_{dR}^2(P)\otimes\frak{g}$ is the curvature form corresponding to $\theta$
and $Q=\G_1\times_{s, \G_0, \pi} P$. The total pseudo-curvature
$\Omega_{\mathrm{total}}$ has therefore two terms. The first term $\partial
\theta:=s^*\theta-t^*\theta$ is a $\frak{g}$-valued $1$-form on $Q$ and the
second term $\Omega$ is a $\frak{g}$-valued $2$-form on $P$. Both terms are of
total degree $2$ in the double complex $\Omega^*(Q_{\bullet}) \otimes\frak{g}$.
As for connections on principal $G$-bundles on smooth manifolds the
pseudo-curvature also satisfies a Bianchi type identity (see \cite{LTX}, Prop.
3.3). 

A pseudo-connection $\theta\in \Omega^1_{dR}(P)\otimes \frak{g}$ is called a
{\it connection} if $\partial\theta=0$ and in this case
$\Omega_{\mathrm{total}}=\Omega\in \Omega_{dR}^2(P)\otimes\frak{g}$ is called
the {\it curvature}. 
A connection $\theta$ is said to be a {\it flat connection} if its curvature
$\Omega$ vanishes.
A necessary and sufficient condition for a pseudo-connection $\theta$ to be a
connection is that $\theta$ is a basic form with respect to the action of the
pseudo-group of local bisections of the groupoid $\G=[\G_1\rightrightarrows
\G_0]$ (see \cite{LTX}, Prop. 3.6).

In general, in contrast to pseudo-connections, connections for principal
$G$-bundles over \'etale groupoids might not always exist as is shown by the
example of the groupoid $[G\rightrightarrows \cdot]$, where $G$ is a Lie group.
The map $G\rightarrow \cdot$ can be considered as a principal $G$-bundle over
$[G\rightrightarrows \cdot]$, where the groupoid acts on $G$ by left
translations. But connections can only exist if $G$ is a discrete group (see
\cite{LTX}, Example 3.12).

As for manifolds connections behave well with respect to pull-backs under
generalized homomorphisms of \'etale groupoids. If $\G$ and $\G'$ are Morita
equivalent groupoids, there is an equivalence of categories of principal
$G$-bundles with connections over $\G$ and $\G'$. Similarly there is an
equivalence of categories of principal $G$-bundles with flat connections over
$\G$ and $\G'$. Therefore, one can also speak of connections and flat
connections of principal $G$-bundles over differentiable stacks, though
connections might not always exist. The groupoid example above corresponds to
the classifying stack $\BG G$ of the Lie group $G$ and so principal $G$-bundles
over $\BG G$ don't have connections unless $G$ is a discrete group. But we note
the following important case where connections do exist \cite{LTX}, Theorem
3.16:

\begin{proposition} Let $G$ be a Lie group. Any principal $G$-bundle over a
proper \'etale groupoid $\G$ admits a connection. In particular, principal
$G$-bundles over orbifolds admit a connection.
\end{proposition}

Using the correspondence between the category of principal $G$-bundles over an
\'etale groupoid 
$\G=[\G_1\rightrightarrows \G_0]$ with the category of principal $G$-bundles
over the  associated simplicial smooth manifold $\G_{\bullet}$ a connection
$\theta$ on a principal $G$-bundle over $\G$ corresponds to a connection, also
be denoted by $\theta$, on the principal $G$-bundle
$\pi_{\bullet}: P_{\bullet}\rightarrow \G_{\bullet}$ over the associated
simplicial smooth manifold
 $\G_{\bullet}$, i. e.  a $G$-invariant $1$-form in the de Rham complex of
compatible forms
$$\theta \in \Omega^1_{dR}(P_{\bullet})\otimes \frak{g}$$
taking values in the Lie algebra $\frak{g}$ of $G$, on which $G$ acts via the
adjoint representation,
such that for each $n$ the restriction
$$ \theta^{(n)}=\theta|_{\Delta^n\times P_n},$$
is a connection on the bundle
$\pi_n: \Delta^n\times P_n \rightarrow \Delta^n\times \G_n$.
In this way $\theta=\{\theta^{(n)}\}$ can as well be interpreted as a sequence
of
$\mathfrak g$-valued compatible 1-forms.

The curvature $\Omega$ of a connection form $\theta$ is then given as the
differential form
$$\Omega=d\theta + \frac{1}{2}[\theta, \theta]\in
\Omega^2_{dR}(\G_{\bullet})\otimes \frak{g}.$$

The behaviour of connection forms under invariant polynomials with respect to
the Chern-Weil map for the associated 
simplicial smooth manifold is given as follows (see \cite{D1}, Proposition 3.7)

\begin{theorem}
Let $G$ be a Lie group and $\Phi$ be an invariant polynomial. The differential
form
$\Phi(\theta)\in \Omega^*_{dR}(P_{\bullet})$ is a closed form and
descends to a closed form in $\Omega^*_{dR}(\G_{\bullet})$ and its
cohomology class represents the image of the class $\Phi\in H^*(BG;
\C)$ under the Chern-Weil map
$H^*(BG; \C) \rightarrow H^*(||\G_{\bullet}||; \C)$ associated to the principal
$G$-bundle $\pi_{\bullet}: P_{\bullet}\rightarrow \G_{\bullet}$, where $BG$ is
the classifying space of $G$.
\end{theorem}

We will denote the associated form in $\Omega^*_{dR}(\G_{\bullet})$ also by
$\Phi(\theta)$. When defining characteristic classes it is necessary that given
any connection on a principal bundle, we can construct a connection on (a model
of) the
universal bundle that pulls back to the given one. This follows from the
following theorem:

\begin{theorem}\label{uni}
Let $G$ be a Lie group and $\G=[\G_1\rightrightarrows \G_0]$ be an \'etale
groupoid. Let $P\stackrel{\pi}\rightarrow \G_0$ be a principal $G$-bundle over
$\G$ with connection. 
Then there exists a bisimplicial smooth manifold $B_{\bullet
\bullet}$ of the homotopy type of the classifying space $BG$ and a
$G$-principal bundle $U_{\bullet \bullet}\rightarrow B_{\bullet
\bullet}$ with a connection $\theta_{U_{\bullet \bullet}}\in
\Omega^1_{dR}(U_{\bullet \bullet}; \frak{g})$ and a morphism $(\Psi,
\psi)$ of $G$-bundles
\[
\xy
\xymatrix{
P_{\bullet}\ar[r]^{\Psi}
\ar[d]&U_{\bullet \bullet}\ar[d]\\
\G_{\bullet}\ar[r]^{\psi}&B_{\bullet \bullet}}
\endxy
\]
such that $\Psi^*(\theta_{U_{\bullet \bullet}})=\theta$.
\end{theorem}

\begin{proof} This is basically \cite{FN}, Theorem 1.2, (for principal
$GL_n(\C)$-bundles this can also be found in
\cite{DHZ}, Prop. 6.15.) using Chern-Weil theory for bisimplicial smooth
manifolds and the categorical correspondence 
between principal $G$-bundles over an \'etale groupoid $\G$ and principal
$G$-bundles over the associated simplicial smooth manifold $\G_{\bullet}$. 
\end{proof}

\section{Multiplicative Cohomology and  Differential \\ Characters for  \'etale
groupoids}\label{MH}

In this section we will construct general versions of Karoubi's multiplicative
cohomology
and Cheeger-Simons differential characters for \'etale groupoids
with respect to any given filtration of the de Rham complex. We will closely
follow the approach for simplicial smooth manifolds as described in \cite{FN}
and for the convenience of the reader will recall the necessary constructions
and concepts.
As a special case with respect to the `filtration b\^{e}te' we will
also recover the group of Cheeger-Simons differential characters for
\'etale groupoids as discussed in \cite{LU}.

For a given complex of abelian groups $C^*$ let $\sigma_{\geq p}C^*$
denote the filtration via truncation in degrees below $p$ and
similarly let $\sigma_{<p}C^*$ denote the truncation of $C^*$ in degrees
greater or equal $p$. We will first consider the special case of
Deligne's `filtration b\^ete' \cite{De} for the de Rham
complex $\Omega_{dR}^*(\G_{\bullet})$ of a smooth \'etale groupoid
$\G$. The `filtration b\^ete' $\sigma=\{\sigma_{\geq
p}\Omega^*_{dR}(\G_{\bullet})\}$ is given via truncation in degrees
below $p$
$$\sigma_{\geq p}\Omega_{dR}^j(\G_{\bullet}) =
\begin{cases} 0 & j<p, \\ \Omega^j_{dR}(\G_{\bullet}) & j\geq p \end{cases}$$

We define the group of Cheeger-Simons differential characters as follows:

\begin{definition}
Let $\G$ be a smooth \'etale groupoid and $\Lambda$ be a subgroup of
$\C$.
The {\it group of (mod $\Lambda$) differential characters of degree $k$ of}
$\G$ is
given by
$$\hat{H}^{k-1}(\G_{\bullet}; \C/\Lambda)= H^k({\mathrm cone}(\sigma_{\geq k}
\Omega^*_{dR}(\G_{\bullet})\rightarrow C^*(\G_{\bullet};
\C/\Lambda))).$$
\end{definition}

Now let ${\mathcal F}=\{F^r\Omega^*_{dR}(\G_{\bullet})\}$ be any
given filtration of the de Rham complex. We define the
multiplicative cohomology groups of $\G$ with respect to
${\mathcal F}$ as follows:

\begin{definition}
Let $\G$ be a smooth \'etale groupoid, $\Lambda$ be a
subgroup of $\C$ and ${\mathcal
F}=\{F^r\Omega^*_{dR}(\G_{\bullet})\}$ be a filtration of
$\Omega^*_{dR}(\G_{\bullet})$. The {\it groups of multiplicative cohomology 
of $\G$ associated to the filtration ${\mathcal F}$} are
given by
$$MH^{2r}_{n}(\G_{\bullet}; \Lambda; {\mathcal F})=H^{2r-n}({\mathrm cone}
(C^*(\G_{\bullet}; \Lambda)\oplus
F^r\Omega^*_{dR}(\G_{\bullet})\rightarrow C^*(\G_{\bullet}; \C))).$$
\end{definition}

If the \'etale groupoid $\G=[M\rightrightarrows M]$ is just a smooth manifold
$M$ as in Example 1.3, we recover the multiplicative cohomology groups of 
Karoubi \cite{K1}, \cite{K2} and we have $MH^{2r}_{n}(M; \Lambda; {\mathcal
F})\cong MH^{2r}_{n}(\G_{\bullet}; \Lambda; {\mathcal F})$.
If $G$ is a compact Lie group and $\G=[G\times M\rightrightarrows M]$ a
transformation groupoid as in Example 1.4, we will get equivariant versions of
multiplicative cohomology. If $\G$ is a proper \'etale groupoid representing an
orbifold as in Example 1.5, then the groups of multiplicative cohomology of the
associated inertia groupoid $MH^{2r}_{n}(\Lambda(G)_{\bullet}; \Lambda;
{\mathcal F})$ are the corresponding orbifold versions of multiplicative
cohomology.

Using similar arguments as in \cite{TXL}, 2.1 it is possible to show that the
groups of multiplicative cohomology of \'etale groupoids behave well with
respect to generalized homomorphisms and Morita equivalence and under mild
conditions therefore will give interesting invariants for differentiable stacks
and orbifolds, which we aim to study in detail in a follow-up article. A main
aspect here will be to relate them to versions of smooth Deligne cohomology for
differentiable stacks and orbifolds.

We also introduce here a more general version of differential characters for
\'etale groupoids associated to any given filtration of the de Rham complex. 
For smooth manifolds these invariants were studied systematically by the first
author in \cite{F}.

\begin{definition}
Let $\G$ be a smooth \'etale groupoid, $\Lambda$ be a
subgroup of $\C$ and ${\mathcal
F}=\{F^r\Omega^*_{dR}(\G_{\bullet})\}$ be a filtration of
$\Omega^*_{dR}(\G_{\bullet})$. The {\it groups of differential characters
(mod $\Lambda$) of degree $k$ of $\G$ associated to the
filtration ${\mathcal F}$} are given by
$$\hat{H}^{k-1}_r(\G_{\bullet}; \C/\Lambda; {\mathcal F})=
H^k({\mathrm cone}(\sigma_{\geq
k}F^r\Omega^*_{dR}(\G_{\bullet})\rightarrow C^*(\G_{\bullet};
\C/\Lambda))).$$
\end{definition}

If ${\mathcal F}$ is Deligne's `filtration b\^ete' of
$\Omega^*_{dR}(\G_{\bullet})$, we recover the ordinary groups
of differential characters of $\G$ as in Definition 2.1. And again as with the
more general groups of multiplicative cohomology, looking at the \'etale
groupoids in Examples 1.3-1.5 we will recover the classical Cheeger-Simons
differential characters \cite{CS} in the case of smooth manifolds, equivariant
versions in the case of a transformation groupoid \cite{G} and in the case of a
proper \'etale groupoid orbifold versions of differential characters associated
to the inertia groupoid \cite{LU}.

The following theorem generalizes \cite{F}, Theorem 2.3 for smooth manifolds:

\begin{theorem}
Let $\G$ be a smooth \'etale groupoid, $\Lambda$ be a
subgroup of $\C$ and ${\mathcal
F}=\{F^r\Omega^*_{dR}(\G_{\bullet})\}$ be a filtration of
$\Omega^*_{dR}(\G_{\bullet})$. There exists a surjective map
$$\Xi:\hat{H}^{2r-n-1}_r(\G_{\bullet}; \C/\Lambda; {\mathcal F})\rightarrow
MH^{2r}_n(\G_{\bullet};\Lambda; {\mathcal F})$$ whose kernel is the
group of forms in $F^r\Omega_{dR}^{2r-n-1}(\G_{\bullet})$ modulo
those forms that are closed and whose complex cohomology class is
the image of a class in $H^*(\G_{\bullet};\Lambda)$.
\end{theorem}

\noindent{\bf Proof:} We denote by ${\mathcal A}(F^r)$ and ${\mathcal B}(F^r)$ 
the cone complexes used in the definition of the groups of differential
characters and multiplicative cohomology associated to the filtration
${\mathcal F}$, i.e.
$${\mathcal A}(F^r)={\mathrm cone}(\sigma_{\geq
k}F^r\Omega^*_{dR}(\G_{\bullet})\rightarrow C^*(\G_{\bullet}; \C/\Lambda))$$
$${\mathcal B}(F^r)={\mathrm cone}
(C^*(\G_{\bullet}; \Lambda)\oplus
F^r\Omega^*_{dR}(\G_{\bullet})\rightarrow C^*(\G_{\bullet}; \C))$$
There is a quasi-isomorphism between the cone complexes
$${\mathrm cone}(\sigma_{\geq k}F^r\Omega^*_{dR}(\G_{\bullet})\rightarrow
C^*(\G_{\bullet}; \C/\Lambda))$$
$${\mathrm cone}(C^*(\G_{\bullet}; \Lambda)\oplus\sigma_{\geq
k}F^r\Omega^*_{dR}(\G_{\bullet})\rightarrow
C^*(\G_{\bullet}; \C))$$
and we get a short exact sequence of complexes
$$0\rightarrow {\mathcal A}(F^r)\rightarrow {\mathcal B}(F^r)\rightarrow
\sigma_{<k}F^r\Omega^*_{dR}(\G_{\bullet})\rightarrow 0$$ where
$\sigma_{<k}$ denotes truncation in degrees greater or equal to $k$.
Everything follows now from the long exact sequence in cohomology
associated to this short exact sequences of complexes, because for
$k=2r-n$ the cohomology group
$H^{2r-n}(\sigma_{<2r-n}F^r\Omega^*_{dR}(\G_{\bullet}))$ is trivial.
\qed

We can also identify the Cheeger-Simons differential characters
with multiplicative cohomology groups in the following way

\begin{corollary}
Let $\G$ be a smooth \'etale groupoid and $\Lambda$ be a subgroup of
$\C$. There is an isomorphism
$$\hat{H}^{r-1}(\G_{\bullet}; \C/\Lambda)\cong
MH^{2r}_r(\G_{\bullet};\Lambda; \sigma)$$
\end{corollary}

\noindent {\bf Proof.} This is a direct consequence of Theorem 2.4 for the case
$n=r$ and the filtration
${\mathcal F}$ is just Deligne's `filtration b\^ete'  \qed

Similarly as for smooth manifolds, it can be shown that the multiplicative
cohomology
groups for smooth \'etale groupoids fit into a long exact sequence (compare
\cite{K3}).

\begin{equation*}
\begin{split}
\ldots\rightarrow H^{2r-n-1}(||\G_{\bullet}||;\Lambda)\rightarrow
H^{2r-n-1}(\Omega^*_{dR}(\G_{\bullet}/F^r\Omega^*_{dR}(\G_{\bullet})))
\rightarrow & \\
\rightarrow  MH^{2r}_{n}(\G_{\bullet}; \Lambda; {\mathcal F})\rightarrow \ldots
\end{split}
\end{equation*}

Following from the above it can be shown as in \cite{CS} that the groups of
differential characters also fit into short exact sequences.

\section{Multiplicative K-theory for \'etale groupoids}\label{MK}

Let $\G=[\G_1\rightrightarrows \G_0]$ be an \'etale groupoid. 
Further let  $G$ be a Lie group with Lie algebra $\frak{g}$. Assume we are given
a principal $G$-bundle on $\G$ with connection forms 
$\theta_0,\ldots,\theta_q$ giving associated connection forms
$\theta_0,\ldots,\theta_q$ on the
principal $G$-bundle $\pi_{\bullet}: P_\bullet\rightarrow \G_\bullet$,  i.e.
$$\theta_j\in \Omega^1_{dR}(P_\bullet)\otimes{\mathfrak g}$$
such that for all $p$ and all $0\leq j\leq q$
$$\theta^{(p)}_j\in \Omega^1_{dR}(\Delta^p\times P_p)\otimes{\mathfrak g}$$
i.e. the restrictions $\theta^{(p)}_j$ are connections on the bundle
$$\Delta^p\times P_p\rightarrow \Delta^p\times \G_p.$$
Fix $q$ and let $\Delta^q$ be the standard simplex
in $\R^{q+1}$ parametrized by coordinates $(s_0,\ldots,s_q)$.

\begin{lemma}
The form $\sum_{j=0}^q\theta_js_j$
defines a connection on the pullback
bundle $\pi^*P_\bullet\rightarrow  \G_\bullet\times \Delta^q$,
where $\pi: \G_\bullet\times \Delta^q \rightarrow \G_\bullet$ is
the projection.
\end{lemma}

\noindent {\bf Proof.} For each $m$ the sum $(\sum_{j=0}^q\theta_js_j)^{(m)}=
\sum_{j=0}^q\theta_j^{(m)}s_j$
is a connection on the bundle
$$\Delta^m\times P_m\times\Delta^q\rightarrow  \Delta^m\times
\G_m\times\Delta^q.$$
We have to verify that the compatibility conditions hold.
The strict simplicial structure on $\G_\bullet \times \Delta^q$ is given by
the maps $\epsi_i^\prime=\epsi_i\times id_{\Delta^q}$ for all $i$,
where  $\epsi_i$  is the map given by the simplicial structure on $\G_\bullet$.
We have
$$(\epsi^i\times
id_{\G_m\times\Delta^q})^*(\sum_{j=0}^q\theta^{(m)}s_j)=\sum_{j=0}
^q(\epsi^i\times id_{\G_m})^*\theta_j^{(m)}s_j$$
since the forms $\theta_j^{(m)}s_j$ are in
$$\Omega^1_{dR}(\Delta^m\times P_m; {\frak g})\otimes
\Omega_{dR}^0(\Delta^q)\subset \Omega^1_{dR}(\Delta^m\times
P_m\times\Delta^q).$$
Now, since the $\theta_j$
satisfy the compatibility conditions we have
$$\sum_{j=0}^q(\epsi^i \times
id_{\G_m})^*\theta_j^{(m)}s_j=\sum_{j=0}^q(id_{\Delta^{m-1}}\times
\epsi_i)^*\theta_j^{(m-1)}s_j.$$
As before we have
$$\sum_{j=0}^q(id_{\Delta^{m-1}}\times
\epsi_i)^*\theta_j^{(m-1)}s_j=(id_{\Delta^{m-1}}\times
\epsi_i^\prime)^*(\sum_{j=0}^q\theta^{(m-1)}s_j),$$ which proves the lemma \qed

Given  an invariant polynomial $\Phi$ of degree $k$, we denote by
$$\tilde{\Theta}_q(\Phi;\theta_0,\ldots,\theta_q)\in\Omega^{2k}_{dR}(
\G_\bullet\times\Delta^q)$$
the characteristic form on $\G_\bullet$ associated to $\Phi$ for the curvature
of the
connection $\sum_{j=0}^q\theta_js_j$.
Whenever $\Phi$ is understood, we will omit
it from the notation for the above form.
The closed form  $\tilde{\Theta}_q(\theta_0,\ldots,\theta_q)$ is
a family of compatible closed forms
$$\tilde{\Theta}_q^{(m)}(\theta_0,\ldots,\theta_q) \in
\Omega_{dR}^{2k}(\Delta^m\times \G_m\times\Delta^q).$$
We define a form
$\Theta_q(\theta_0,\ldots,\theta_q)\in\Omega_{dR}^{2k-q}(\G_\bullet)$
by integration
$$\Theta_q(\theta_0,\ldots,\theta_q)=\int_{\Delta^q}\tilde{\Theta}_q(\theta_0,
\ldots,\theta_q),$$
i.e. $\Theta_q(\theta_0,\ldots,\theta_q)$ is the family of forms
$$\Theta_q^{(m)}(\theta_0,\ldots,\theta_q)=\int_{\Delta^q}\tilde{\Theta}_q^{(m)}
(\theta_0,\ldots,\theta_q).$$
These forms satisfy the compatibility conditions since the diagram
\[\xymatrix{
\Omega_{dR}^*(\Delta^m\times
\G_m\times\Delta^q)\ar[r]^{\int_{\Delta^q}}\ar[d]_{(\epsi^i\times\mathrm{id}_{
\G_m\times
\Delta^q})^*}&
\Omega_{dR}^*(\Delta^m\times \G_m)\ar[d]^{(\epsi^i\times\mathrm{id}_{\G_m})^*}\\
\Omega_{dR}^*(\Delta^{m-1}\times
\G_m\times\Delta^q)\ar[r]^{\int_{\Delta^q}}&
\Omega_{dR}^*(\Delta^{m-1}\times \G_m)\\
\Omega_{dR}^*(\Delta^{m-1}\times
\G_{m-1}\times\Delta^q)\ar[r]^{\int_{\Delta^q}}\ar[u]^{(id_{\Delta^{m-1}}\times
\epsi_i^\prime)^*}& \Omega_{dR}^*(\Delta^{m-1}\times
\G_{m-1})\ar[u]_{(id_{\Delta^{m-1}}\times \epsi_i)^*}}\] commutes and
the forms $\tilde{\Theta}_q^{(m)}(\theta_0,\ldots,\theta_q) \in
\Omega^{2k}_{dR}(\Delta^m\times \G_m\times\Delta^q)$ are compatible.

If we denote by $t$ the variables on the simplices $\Delta^p$, by
$x$ the variables on the smooth manifolds $\G_p$ and by $s$ the variables on
the simplex $\Delta^q$, then we can write the differential on the complex
$\Omega^*_{dR}(\Delta^q\times
\G_\bullet)$ as $\d=\d_s + \d_{t,x}$, where $\d_{t,x}$ is the
differential of the complex $\Omega^*_{dR}(\G_\bullet)$. Since
$\tilde{\Theta}_q(\theta_0,\ldots,\theta_q)$ is closed, we have
\[\d_{t,x}\tilde{\Theta}_q(\theta_0,\ldots,\theta_q)=-\d_s\tilde{\Theta}
_q(\theta_0,\ldots,\theta_q).\]
Then we have
\[\d_{t,x}\Theta_q(\theta_0,\ldots,\theta_q)=\d_{t,x}\int_{\Delta^q}\tilde{
\Theta}_q(\theta_0,\ldots,\theta_q)=\]
\[=
\int_{\Delta^q}\d_{t,x}\tilde{\Theta}_q(\theta_0,\ldots,\theta_q)=-
\int_{\Delta^q}\d_s\tilde{\Theta}_q(\theta_0,\ldots,\theta_q).\]
By Stokes' theorem the last integral is equal to
$-\int_{\partial\Delta^q}\tilde{\Theta}_q(\theta_0,\ldots,\theta_q)$, so we have
proven the analogue of Theorem 3.3 in \cite{K2} for  smooth \'etale groupoids:

\begin{proposition}\label{st}Let $\G$ be a smooth \'etale groupoid. In the de
Rham complex $\Omega^*_{dR}(\G_\bullet)$ of compatible forms we have
\[\d\Theta_q(\theta_0,\ldots,\theta_q)=-\sum_{i=0}^q(-1)^i\Theta_{q-1}(\theta_0,
\ldots,\hat{\theta_i},\ldots\theta_q).\]
\end{proposition}

In particular, for $q=1$ we have that if given any two connections $\theta_0$
and $\theta_1$ on $P_\bullet$ and an invariant polynomial $\Phi$, we can write
canonically
\[\Phi(\theta_1)-\Phi(\theta_0)=\d\Theta_1(\Phi;\theta_0,\theta_1).\]

Write a formal series of invariant polynomials $\Phi$ as
a sum $\sum_r\Phi_r$ with $\Phi_r$ a homogeneous polynomial of
degree $r$ (see \cite{K2},\cite{K3}). 
Let ${\mathcal F}=\{F^r\Omega^*_{dR}(\G_{\bullet})\}$ be again
a filtration of the de Rham complex of $\G_\bullet$ and
$$\omega=\sum_r\omega_r,\,\,\,\eta=\sum_r\eta_r$$
be formal sums of forms in $\Omega_{dR}^*(\G_\bullet)$ (note that we
do not require that $\omega_r$ is of degree $r$, actually most of
the times this will not be the case). We will write $\omega=\eta
\mbox{ mod } {\mathcal F}$ if and only if for each $r$ we have
\begin{equation}\label{req}
\omega_r-\eta_r\in F^r\Omega_{dR}^*(\G_\bullet).
\end{equation}
We will also write $\omega=\eta \mbox{ mod } \tilde{{\mathcal F}}$
when for each $r$ the above equation is satisfied modulo exact forms.
We will not distinguish between an invariant polynomial or a formal series in
what follows,
writing just $\Phi$ and $\omega$ also for formal sums. Let us define the notion
of a multiplicative bundle:

\begin{definition}
Let $\G$ be a smooth \'etale groupoid and $\Phi$ be an invariant polynomial (or
a formal series) and
${\mathcal F}=\{F^r\Omega^*_{dR}(\G_{\bullet})\}$ be a filtration of the
de Rham complex of $\G$. An {\it $({\mathcal
F},\Phi)$-multiplicative bundle over} $\G$ is a
triple $(P_\bullet,\theta,\omega)$ where $P_\bullet$ is a principal
$G$-bundle over $\G_\bullet$, $\theta$ is a connection on $P_\bullet$
and $\omega$ is a formal series of forms in
$\Omega^*_{dR}(\G_\bullet)$ such that
\[\Phi(\theta)=\d\omega \mbox{ mod }{\mathcal F}.\]
An {\it isomorphism}
$f:(P_\bullet,\theta,\omega)\rightarrow
(P_\bullet^\prime,\theta^\prime,\omega^\prime)$
between two multiplicative bundles over $\G$ is an isomorphism $f$ of the
underlying bundles
$P_\bullet,P_\bullet^\prime$ such that
\[\omega^\prime-\omega=\Theta_1(\theta,f^*\theta^\prime) \mbox{ mod }
\tilde{{\mathcal F}}.\]
\end{definition}

It follows as in \cite{K2} that this defines an equivalence relation on
multiplicative bundles, so we can
make the following definition:

\begin{definition} Let $\G$ be a smooth \'etale groupoid. The set
$MK^\Phi(\G_\bullet;{\mathcal F})$ 
of isomorphism classes of  $({\mathcal F},\Phi)$-multiplicative bundles is
called the {\it multiplicative K-theory of $\G$
with respect to} $({\mathcal F},\Phi)$.
\end{definition}

If there is no risk of ambiguity, we will omit $\Phi$ and ${\mathcal F}$ from
the notation.

\section{Characteristic classes for secondary theories associated  to \'etale
groupoids}

Let $G$ be a Lie group with Lie algebra $\frak{g}$. Given a principal $G$-bundle
with a connection $\theta$ on the simplicial
smooth manifold $\G_{\bullet}$ associated to an \'etale groupoid $\G$ and an
invariant polynomial $\Phi$ of homogeneous degree $k$ we will associate
characteristic classes with values in multiplicative cohomology
groups and in groups of differential characters of $\G_{\bullet}$
associated to any filtration ${\mathcal F}$ of the simplicial de Rham complex
$\Omega^*_{dR}(\G_{\bullet})$. This generalizes Karoubi's secondary
characteristic classes as constructed in \cite{K2} to smooth \'etale groupoids.
An approach with a somehow similar flavour can also be found in \cite{CM2}.

Let ${\mathcal F}=\{F^r\Omega^*_{dR}(\G_{\bullet})\}$ be a filtration of the de
Rham
complex of $\G$ and $\Gamma=(P_{\bullet}, \theta, \eta)$ a
$({\mathcal F},\Phi)$-multiplicative bundle over $\G$.

The connection $\theta$ on the principal $G$-bundle $\pi_{\bullet}$
is given as a $1$-form $$\theta\in \Omega^1_{dR}(P_{\bullet})\otimes
\frak{g}.$$ The characteristic form of Theorem 1.7
$$\Phi(\theta)\in \Omega^{2k}_{dR}(\G_{\bullet})$$
can also be seen as a family of forms
$$\Phi(\theta^{(n)})\in \Omega^{2k}_{dR}(\Delta^n\times \G_n) \otimes\frak{g}$$
satisfying the compatibility conditions.

Since $\Gamma$ is a multiplicative bundle
we have
$$\Phi(\theta)=d\eta + \omega$$
where the  forms $\eta$ and $\omega$ are compatible sequences
$\eta=\{\eta^{(n)}\}$ and $\omega=\{\omega^{(n)}\}$ of differential
forms with $\omega\in F^r \Omega^{2k}_{dR}(\G_{\bullet})$ and
$\eta\in \Omega^{2k-1}_{dR}(\G_{\bullet})$.

The connection $\theta$ is the pullback of
a connection $\theta_{U_{\bullet \bullet}}$ on $U_{\bullet \bullet}$
by a map $\Psi$ as  in Theorem \ref{uni}.
Let $\Lambda$ be a subring of the complex numbers $\C$,
and assume that $\Phi$ corresponds under the Chern-Weil map to
a $\Lambda$-valued cohomology class.

For every $n$ the inclusion $\imath_n:B_{n\bullet}\rightarrow
B_{\bullet\bullet}$ induces isomorphisms in cohomology since
$||B_{n\bullet}||$ is homotopy equivalent to the classifying space $BG$
of $G$. We also have that
$\imath_n^*\theta_{U_{\bullet\bullet}}= \theta_{U_{n\bullet}}$ for every $n$.
Since the form $\Phi(\theta_{U_{n\bullet}})$ represents the class of
$\Phi$ by Theorem 1.7, and
$\imath_n^*\Phi(\theta_{U{\bullet\bullet}})=\Phi(\theta_{U_{n\bullet}})$,
we have that the form $\Phi(\theta_{U{\bullet\bullet}})\in
\Omega^{*}_{dR}(B_{\bullet\bullet})$ represents the class of $\Phi$.
Then it follows that there exist a compatible cocycle $c\in
C^{2k}(B_{\bullet\bullet};\Lambda)$ and a compatible cochain $v\in
C^{2k-1}(B_{\bullet\bullet};\C)$  such that we have
 $$ \delta v=c-\Phi(\theta_{U{\bullet\bullet}}) $$
 
Since $\Psi^*$ maps compatible cochains (in the bisimplical sense) to
compatible chains (in the simplicial sense),
the triple $\xi(\Gamma)=(\Psi^*(c), \omega, \Psi^*(v)+\eta)$ defines a cocycle
in
the cone complex
$${\mathrm cone}(C^*(\G_{\bullet}; \Lambda)\oplus
F^r\Omega^*_{dR}(\G_{\bullet})\rightarrow C^*(\G_{\bullet}; \C))$$
and since $\omega$ is a form of degree $2k$ also a cocycle in the cone complex
$${\mathrm cone}(C^*(\G_{\bullet}; \Lambda)\oplus\sigma_{\geq
2k}F^r\Omega^*_{dR}(\G_{\bullet})\rightarrow  C^*(\G_{\bullet}; \C)).$$

The triple $\xi(\Gamma)$ is a cocycle, because we have
$\delta\Psi^*c=\Psi^*\delta c=0$. Since $c$ is a cocycle, we have that
$\d\omega=\d\Phi(\theta)+ \d^2\eta=0$,
and also
\[
\delta\Psi^*v  =  \Psi^*\delta v
                              =  \Psi^*(c-\Phi(\theta_{U_{\bullet\bullet}}))
                              =  \Psi^*c-\Phi(\theta)
                              =  \Psi^*c-(\omega+ \d\eta).
\]

The class of $\xi(\Gamma)$ is independent of the choices of $c$ and $v$:
If $c^\prime$ and $v^\prime$ are other choices, we must have $c-c^\prime=\delta
u$ and $\delta u=\delta(v-v^\prime)$.
Then, since $H^{2k-1}(||B_{\bullet\bullet}||;\C)$ is trivial, there exists
a compatible cochain $w$ such that $\delta w=u+(v-v^\prime)$. If
$\xi^\prime(\Gamma)$
is the cocycle obtained from the different choice, then
$\xi(\Gamma)-\xi^\prime(\Gamma)=
(\Psi^*\delta u,0,\Psi^*(v-v^\prime))=\d(\Psi^*u,0,\Psi^*w)$.

Hence for  $2r-m=2k$ we can define the class of the multiplicative bundle
$(P_{\bullet}, \theta, \eta)$ in the
multiplicative cohomology group $MH^{2r}_m (\G_{\bullet}, \Lambda, {\mathcal
F})$
 to be  the class of $\xi(\Gamma)$.
Similarly the class of $(P_{\bullet}, \theta, \eta)$ in
$\hat{H}^{2k-1}_r (\G_{\bullet}; \C/\Lambda; {\mathcal F})$ is the class
of the triple $\xi(\Gamma)$.

\begin{proposition}
The classes constructed above are characteristic classes of elements
of $MK^\Phi(\G_\bullet;{\mathcal F})$.
\end{proposition}

\noindent {\bf Proof.} The naturality follows from the construction.
We show that for two isomorphic multiplicative bundles
$\Gamma=(P_{\bullet},\theta,\eta)$  and
$\Gamma^\prime=(P_{\bullet}^\prime,\theta^\prime,\eta^\prime)$ the
cocycles $\xi(\Gamma)$ and $\xi(\Gamma^\prime)$ are  cohomologous.
We can assume $P_\bullet=P_\bullet^\prime$, and write
$\Phi(\theta)=\omega +\d\eta$ and $\Phi(\theta^\prime)=\omega^\prime
+\d\eta^\prime$ with $\omega,\omega^\prime \in F^r
\Omega^{2k}_{dR}(\G_{\bullet})$. Since the two multiplicative bundles
are isomorphic we have
\[\eta^\prime-\eta=\Theta_1(\Phi;\theta,\theta^\prime)+\sigma+\d\rho\]
with $\sigma\in F^r \Omega^{2k-1}_{dR}(\G_{\bullet})$. It follows
that
$$\omega^\prime-\omega=\d(\Theta_1(\Phi;\theta,
\theta^\prime)-(\eta^\prime-\eta))
=-\d(\sigma+\d\rho).$$ Let $\Psi^\prime$ be the map pulling back
$\Gamma^\prime$ given by Theorem 1.8 and let $c^\prime,v^\prime$
be  the cochains used in the construction for the characteristic
cycle $\xi(\Gamma^\prime)$. Then
$$\xi(\Gamma^\prime)-\xi(\Gamma)=(\Psi^{\prime *}c^\prime-\Psi^*c,
\omega^\prime-\omega,
\Psi^{\prime *}v^\prime-\Psi^*v+\Theta_1(\Phi;\theta,\theta^\prime)+\sigma
+\d\rho)$$
is cohomologous to the triple $\zeta=(\Psi^{\prime *}c^\prime-\Psi^*c,0,
\Psi^{\prime *}v^\prime-\Psi^*v+\Theta_1(\Phi;\theta,\theta^\prime))$
since the two differ only by the coboundary of $(0,-\sigma,\rho)$.
We can choose $c^\prime=c$ and
$v^\prime=v+\Theta_1(\Phi;\theta_{U{\bullet\bullet}},\theta_{U{\bullet\bullet}}
^\prime)$,
where $\theta_{U{\bullet\bullet}}^\prime$ is the connection pulling back
to $\theta^\prime$ under $\Psi^{\prime *}$ given by Theorem 1.8. Hence we have,
using the naturality of
the first transgression form,
$$\zeta=(\Psi^{\prime *}c-\Psi^*c,0,
\Psi^{\prime *}v-\Psi^*v+\Theta_1(\Phi;\Psi^{\prime
*}\theta_{U{\bullet\bullet}},\theta^\prime)
+\Theta_1(\Phi;\theta,\theta^\prime)).$$
Proposition 3.2 now implies that
\begin{equation*}
\begin{split}
\Theta_1(\Phi;\Psi^*\theta_{U{\bullet\bullet}},\Psi^{\prime
*}\theta_{U{\bullet\bullet}})
+\d\Theta_2(\Phi;\Psi^*\theta_{U{\bullet\bullet}},\Psi^{\prime
*}\theta_{U{\bullet\bullet}},\theta^\prime)=\\
\Theta_1(\Phi;\Psi^{\prime *}\theta_{U{\bullet\bullet}},\theta^\prime)
+\Theta_1(\Phi;\Psi^*\theta_{U{\bullet\bullet}},\theta^\prime).
\end{split}\end{equation*}
But $\Psi^\prime$ and $\Psi$ are homotopic, so there is a chain homotopy
$H$ between the induced cochain maps and using $H$ we can write $\zeta$ as
$$(\delta Hc,0, \delta Hv +H\delta v+
\Theta_1(\Phi;\Psi^*\theta_{U{\bullet\bullet}},\Psi^{\prime
*}\theta_{U{\bullet\bullet}})
+\d\Theta_2(\Phi;\Psi^*\theta_{U{\bullet\bullet}},\Psi^{\prime
*}\theta_{U{\bullet\bullet}},\theta^\prime)).$$
Therefore $\zeta$ is cohomologous to  $(\delta Hc,0, H\delta v+
\Theta_1(\Phi;\Psi^*\theta_{U{\bullet\bullet}},\Psi^{\prime
*}\theta_{U{\bullet\bullet}}))$, and since the transgression forms
$\Theta_1(\cdot;\cdot,\cdot)$ are compatible with chain homotopies
(see \cite{DHZ}, appendix A), the former cocycle  is cohomologous to
$$(\delta Hc,0,H\delta
v+H\Phi(\theta_{U{\bullet\bullet}}))=\d(Hc,0,0)$$ because
$Hc=H(\delta v +\Phi(\theta_{U{\bullet\bullet}}))$.\qed

\end{document}